\documentclass[11pt]{amsart}
\usepackage[margin=1.35in]{geometry}
\usepackage{helvet, color}
\usepackage{amscd,amsmath,amsxtra,amsthm,amssymb,stmaryrd,xr,mathrsfs,mathtools,enumerate,commath, comment}
\usepackage[dvipsnames]{xcolor}
\usepackage{hyperref}
\hypersetup{
 colorlinks=true,
 linkcolor=blue,
 filecolor=magenta,
 citecolor=Aquamarine,
 urlcolor=cyan,
 pdftitle={Statistics for Iwasawa invariants of elliptic curves, II},
 }
 \usepackage{tikz-cd}
\usepackage[dvipsnames]{xcolor}
\usepackage[all,cmtip]{xy}

\DeclareFontFamily{U}{wncy}{}
\DeclareFontShape{U}{wncy}{m}{n}{<->wncyr10}{}
\DeclareSymbolFont{mcy}{U}{wncy}{m}{n}
\DeclareMathSymbol{\Sh}{\mathord}{mcy}{"58}

\theoremstyle{plain}
 \theoremstyle{definition}
\newtheorem{theorem}{Theorem}[section]
\newtheorem{lemma}[theorem]{Lemma}
\newtheorem{question}[theorem]{Question}
\newtheorem{conj}[theorem]{Conjecture}

\newtheorem{proposition}[theorem]{Proposition}
\newtheorem{corollary}[theorem]{Corollary}

\numberwithin{equation}{section}
\newtheorem{definition}[theorem]{Definition}

\theoremstyle{remark}

\newcommand{\T}{\operatorname{T}}
\newcommand{\I}{\operatorname{I}}

\newcommand{\Gal}{\operatorname{Gal}}

\newcommand{\Reg}{\operatorname{Reg}}

\newcommand{\FF}{\mathbb{F}}

\newcommand{\Z}{\mathbb{Z}}

\newcommand{\Q}{\mathbb{Q}}
\newcommand{\F}{\mathbb{F}}

\newcommand{\Hom}{\mathrm{Hom}}

\newcommand{\Sel}{\mathrm{Sel}}

\newcommand{\rank}{\mathrm{rank}}

\newcommand{\corank}{\mathrm{corank}}

\usepackage[normalem]{ulem}

\newcommand{\op}[1]{\operatorname{#1}}
\newcommand{\Selp}{\Sel_{p^{\infty}}(E/\mathbb{Q}_{\infty})}
\newcommand{\mup}{\mu_p(E/\Q_{\infty})}
\newcommand{\lap}{\lambda_p(E/\Q_{\infty})}
\newcommand{\gp}{g_p(E/\Q_{\infty})}

\newcommand{\Echi}{\mathscr{E}_{p^n| \chi}}

\begin{document}
\title[Statistics for Iwasawa Invariants of elliptic curves, $\rm{II}$]{Statistics for Iwasawa Invariants of elliptic curves, \rm{II}}

\author[D.~Kundu]{Debanjana Kundu}
\address[Kundu]{Department of Mathematical and Statistical Sciences\\ UTRGV \\ 1201 W University Dr.\\ Edinburg, TX 78539\\ USA}
\email{dkundu@math.toronto.edu}

\author[A.~Ray]{Anwesh Ray}
\address[Ray]{Chennai Mathematical Institute, H1, SIPCOT IT Park, Kelambakkam, Siruseri, Tamil Nadu 603103, India}
\email{anwesh@cmi.ac.in}
\begin{abstract}
We study the average behaviour of the Iwasawa invariants for Selmer groups of elliptic curves.
These results lie at the intersection of arithmetic statistics and Iwasawa theory.
We obtain lower bounds for the density of rational elliptic curves with prescribed Iwasawa invariants.
\end{abstract}


\maketitle

\section{Introduction}
In recent years, research in arithmetic statistics has gained considerable popularity.
A early development in the subject is due to H.~Cohen and H.~W.~Lenstra in \cite{CL84}, where they introduced a plausible heuristic on the distribution of class groups of number fields.
Over the years, such heuristics have been suitably modified and they give a powerful source of predictions for number fields.
One route that arithmetic statistics took was in the direction of elliptic curves, which will be our primary focus.

A fundamental result in the theory of elliptic curves is the \emph{Mordell--Weil Theorem}.
It states that given an elliptic curve $E_{/\Q}$, its $\Q$-rational points form a finitely generated abelian group, i.e.,
\[
E(\Q) \simeq \Z^r \oplus T
\]
where $r$ is a non-negative integer called the \emph{(Mordell--Weil) rank} and $T$ is a finite group, called the \emph{torsion subgroup}.
This torsion subgroup is well-understood and its possible structure is known by the work of B.~Mazur (see \cite{Maz77, Maz78}).
On the other hand, the rank remains mysterious and to-date there is no known algorithm to compute it.
The \emph{rank distribution conjecture} of elliptic curves claims that half of all elliptic curves have rank zero, the remaining half have rank one, and all higher ranks constitute zero percent of elliptic curves (even though there may exist infinitely many such elliptic curves).
Therefore, a suitably-defined \emph{average rank} would be $1/2$.
The best results in this direction are by M.~Bharagava and A.~Shankar (see \cite{BS15_quartic, BS15_cubic}).
They show that the average rank of elliptic curves over $\Q$ is strictly less than one, and that both rank zero and rank one cases comprise non-zero densities across all elliptic curves over $\Q$.
Proving results about the Mordell--Weil rank almost always involves a thorough analysis of the Selmer group.
In \cite{BS13_4Selmer, BS13_5Selmer}, Bhargava--Shankar explicitly computed the average sizes of certain Selmer groups to deduce asymptotic results of Mordell--Weil ranks and made the following conjecture.
\begin{conj}
Let $n$ be any positive integer.
Then, when all elliptic curves $E$ are ordered by height, the average size of the $n$-Selmer group, denoted by $\Sel_n(E/\Q)$, is $\sigma(n)$, the sum of the divisors of $n$.
\end{conj}
\noindent This conjecture has been verified for $n= 2,3,4,5$ and was enough to deduce powerful partial results for the rank distribution conjecture.

Another subject of active research is Iwasawa theory, which was introduced in the late 1950's.
It started as the study of class groups over infinite towers of number fields (see \cite{Iwa59_GammaExtensions, Iwa59_cyclotomic}).
In \cite{Maz72}, Mazur initiated the study of Iwasawa theory of Selmer groups of elliptic curves.
This was motivated by the close relationship between the group of units of number fields and the rational points elliptic curves.
Over the cyclotomic $\Z_p$-extension of $\Q$, Mazur conjectured that for an elliptic curve with good \emph{ordinary} reduction at an odd prime $p$, the $p$-primary Selmer group is cotorsion as a module over the Iwasawa algebra.
This conjecture is now settled by the work of K.~Kato, see \cite[Theorem 17.4]{kato2004p}.

The Iwasawa algebra (denoted by $\Lambda$) is isomorphic to the power series ring $\Z_p\llbracket T\rrbracket$.
The algebraic structure of the Selmer group (as a $\Lambda$-module) is encoded by Iwasawa invariants, $\mu$ and $\lambda$.
The $p$-adic Weierstrass Preparation Theorem asserts that the characteristic ideal of the Pontryagin dual of the Selmer group is generated by a unique element $f_E^{(p)}(T)$, which can be expressed as a power of $p$ times a distinguished polynomial.
Then, the $\mu$-invariant is the power of $p$ dividing $f_E^{(p)}(T)$ and the $\lambda$-invariant is its degree.

In \cite{Maz72}, there are examples of elliptic curves over $\Q$ with good ordinary reduction at $p$ such that the $\mu$-invariant associated to the Selmer group is positive.
On the other hand, there is a long standing conjecture of R.~Greenberg (see \cite[Conjecture 1.11]{greenberg1999iwasawa}) that if $E_{/\Q}$ is an elliptic curve with good \emph{ordinary} reduction at $p$ and $E[p]$ is \emph{irreducible}, then the $\mu$-invariant is zero.
This motivated us to study Iwasawa invariants \emph{on average}, in \cite{KR21}.
The primary goal was to study the following two separate but interrelated problems.
\begin{enumerate}
\item For a fixed elliptic curve $E$, how do the Iwasawa invariants vary as $p$ varies over all odd primes $p$ at which $E$ has good reduction?
\item For a fixed prime $p$, how do the Iwasawa invariants vary as $E$ varies over all elliptic curves (with good reduction at $p$)?
\end{enumerate}
The first question is easier to tackle and was studied by Greenberg in \cite{greenberg1999iwasawa} in the special case of rank 0 elliptic curves with good \emph{ordinary} reduction at $p$.
The second question lies at the intersection of arithmetic statistics and Iwasawa theory, requiring new ideas.

It is well-known that the $\lambda$-invariant associated to the Selmer group of an elliptic curve is \emph{at least} equal to its Mordell--Weil rank (see Lemma \ref{lambdaandrank}).
As a first step, in \cite{KR21}, we developed a technique to distinguish between when the $\lambda$-invariant is \emph{exactly equal} to the Mordell--Weil rank and when it is \emph{strictly greater} than the rank.
This involved using the \emph{truncated Euler characteristic formula} (see Section \ref{truncatedEC} for the precise definition).
This invariant of the elliptic curve is closely related to the $p$-adic Birch and Swinnerton-Dyer (BSD) conjecture and captures information about the size of the Tate--Shafarevich group, the Tamagawa number, the anomalous primes, and the torsion points of the elliptic curve.
Further, we calculated upper bounds for the density of the set of rank 0 elliptic curves defined over $\Q$ whose Euler characteristic is \emph{not} a unit, i.e., when the $\lambda$-invariant is strictly larger than the rank.
Unfortunately, these results are conditional and depend on Cohen--Lenstra type heuristics on the variation of the Tate--Shafarevich group by C.~Delaunay (see \cite{Del01}).
In this paper, we refine our methods and prove Theorem \ref{thm: lower bound}.
Our refined technique allows for better control over the Euler characteristic formula.

In particular, given an odd prime $p$ and a non-negative integer $n$, we can prove an explicit \emph{lower bound} for the proportion of elliptic curves with Mordell--Weil rank \emph{at least 2} or rank \emph{at most 1} with the truncated Euler characteristic divisible by $p^n$.
Since 100\% of the elliptic curves are expected to have rank at most 1, our result (conjecturally) provides a lower bound for the proportion of elliptic curves with rank \emph{at most 1} and the truncated Euler characteristic divisible by $p^n$.

Another important question in the theory of elliptic curves is the \emph{rank boundedness conjecture} which asks whether there is an upper bound for the Mordell--Weil rank of elliptic curves over number fields.
This question is widely open and experts are unable to arrive at a consensus on what to expect, see \cite[Section 3]{PPVW19} for a detailed survey.
Those in favour of unboundedness argue that this phenomenon provably occurs in other global fields, and that the proven lower bound for this upper bound increases every few years.
The current record, is an elliptic curve over $\Q$ with Mordell--Weil rank at least 28, discovered by N.~Elkies.
On the other hand, a recent series of papers by B.~Poonen \emph{et. al.} provides a justified heuristic inspired by ideas from arithmetic statistics which suggests otherwise (see for example, \cite{Poo18}).
Evidently, rank unboundedness is a deep problem.
However, upon restricting our focus to the cyclotomic $\Z_p$-extension of a number field, it is possible to make somewhat precise statements that are analogous.
We have already discussed that the $\lambda$-invariant of the Selmer group of an elliptic curve is closely related to its Mordell--Weil rank.
In Theorem \ref{th64} and Corollary \ref{cor to thm64}, we prove that given any integer $n$, there is an explicit lower bound for the density of the set of elliptic curves with good ordinary reduction at $p$ for which $\lambda+\mu\geq n$.
This lower bound depends on $p$ (and $n$), is strictly positive, and becomes smaller as $p$ or $n$ become larger.
A key ingredient is the observation that a lower bound of the sum $\lambda+ \mu$ depends on the number of primes (distinct from $p$) for which $p$ divides the Tamagawa number and whether $p$ is an anomalous prime (see Lemma \ref{lambda+mu in terms of tamagawa and anomalous}).
In view of Greenberg's conjecture for irreducible residual representation $E[p]$, our results imply that $\lambda\geq n$ for a \emph{positive} proportion of elliptic curves that depends only on $p$ and $n$.
We prove our results assuming the finiteness of the $p$-primary part of the Tate-Shafarevich group $\Sh(E/\Q)$ for elliptic curves $E$ defined over the rationals.
This is indeed expected to be true.

\emph{Organization:} Including this introduction, the article has 7 sections.
Section \ref{section 2: Preliminaries} is preliminary in nature: here, we introduce the Iwasawa invariants and the main object of interest, i.e., the Selmer groups.
In Section \ref{truncatedEC}, we discuss the notion of truncated Euler characteristic.
Next, in Section \ref{densities for W Models}, we arrange elliptic curves (defined over $\Q$) by na{\"i}ve height and calculate what proportion of them have a certain Kodaira type.
In Section \ref{EC on average}, given an odd prime $p$ and a fixed integer $n\geq 0$, we associate to each elliptic curve with good ordinary reduction at $p$, its (truncated) Euler characteristic formula.
We analyze what proportion of elliptic curves have Mordell--Weil rank at most 1 (conjectured to happen 100\% of the time) and the truncated Euler characteristic divisible by $p^n$.
Finally, in Section \ref{Invariants on Average} we state and prove the main theorem of this article.
Given an odd prime $p$ and a fixed integer $n\geq 0$, we prove a lower bound for what proportion of elliptic curves defined over $\Q$ with good ordinary reduction at $p$ is the sum of Iwasawa invariants $\geq n$.
In Section \ref{tables}, we have included some computational tables.

\section{Preliminaries}
\label{section 2: Preliminaries}
Throughout, $p$ will be a fixed odd prime number.
Let $E_{/\Q}$ be an elliptic curve with good ordinary reduction at $p$.
Assume throughout this paper that the $p$-primary part of the Tate-Shafarevich group $\Sh(E/\Q)$ is finite.
This is a well-known conjecture.
The main object of study is the $p$-primary Selmer group of $E$ over the cyclotomic $\Z_p$-extension, which we now introduce.

Fix an algebraic closure $\bar{\Q}$ of $\Q$ and for each prime $\ell$, choose an embedding $\iota_\ell:\bar{\Q}\hookrightarrow \bar{\Q}_\ell$.
Let $S$ be a finite set of prime numbers containing $p$ and the primes at which $E$ has bad reduction, and $\Q_S$ be the maximal algebraic extension of $\Q$ which is unramified at all primes $\ell\notin S$.
Denote by $\op{G}_{S}$ the Galois group $\Gal(\Q_S/\Q)$.
For any discrete $\op{G}_S$-module $N$, let $H^i(\Q_S/\Q, N)$ denote the cohomology group $H^i(\op{G}_S(\Q), N)$.
For $n\geq 0$, let $\Q_n$ be the (unique) subfield of $\Q(\mu_{p^{n+1}})$ of degree $p^n$ over $\Q$.
Note that $\Q_n$ is contained in $\Q_{n+1}$.
Let $\Q_{\infty}$ be the union \[\Q_{\infty}:=\bigcup_{n\geq 0} \Q_n\]and set $\Gamma:=\Gal(\Q_{\infty}/\Q)$.
There are isomorphisms of topological groups \[
\Gal(\Q_{\infty}/\Q)\xrightarrow{\sim} \varprojlim_n\Gal(\Q_{n}/\Q)\xrightarrow{\sim} \Z_p.
\]
The extension $\Q_{\infty}$ is the \emph{cyclotomic $\Z_p$-extension} of $\Q$ and $\Q_n$ is its \textit{$n$-th layer}.
Choose a topological generator $\gamma\in \Gamma$ and fix an isomorphism $\Z_p\xrightarrow{\sim} \Gamma$ sending $a$ to $\gamma^a$.
The Iwasawa algebra $\Lambda$ is defined as the following inverse limit
\[\Lambda:=\varprojlim_n \Z_p[\Gal(\Q_n/\Q)].\] Fix an isomorphism of $\Lambda$ with the ring of formal power series $\Z_p\llbracket T\rrbracket$, upon identifying $\gamma-1$ with $T$.
\subsection{} Denote by $E[p^{\infty}]$ the $\op{G}_S$-module consisting of all $p$-power torsion points in $E(\bar{\Q})$.
For $\ell\in S$, and any finite extension $L/\Q$, define
\[
J_\ell(E/L):= \bigoplus_{w|\ell} H^1\left( L_w, E\right)[p^\infty]
\]
where the direct sum is over all primes $w$ of $L$ lying above $\ell$.
For any field $L\subset \Q_S$, the \emph{$p$-primary Selmer group} over $L$ is defined as follows
\[
\Sel_{p^\infty}(E/L):=\ker\left\{ H^1\left(\Q_S/L,E[p^{\infty}]\right)\xrightarrow{\Phi_{E,L}} \bigoplus_{\ell\in S} J_\ell(E/L)\right\}.
\]
By a well-known exercise involving Nakayama's lemma, the Pontryagin dual \[\Selp^{\vee}:=\Hom(\Selp, \Q_p/\Z_p)\] is finitely generated as a $\Lambda$-module.
It is a deep result of Kato that $\Selp^{\vee}$ is a torsion $\Lambda$-module, see \cite{kato2004p}.

\subsection{}
Let $M$ be a cofinitely generated cotorsion $\Lambda$-module. 
We introduce certain Iwasawa invariants associated to $M$.
According to \cite[Theorem 13.12]{washington1997}, the module $M^{\vee}$ is pseudo-isomorphic to a finite direct sum of cyclic $\Lambda$-modules, i.e., there is a map of $\Lambda$-modules
\[
M^{\vee}\longrightarrow \left(\bigoplus_{i=1}^s \Lambda/(p^{\mu_i})\right)\oplus \left(\bigoplus_{j=1}^t \Lambda/(f_j(T)) \right)
\]
with finite kernel and cokernel.
Here, $\mu_i>0$ and $f_j(T)$ is a distinguished polynomial (i.e., a monic polynomial with non-leading coefficients divisible by $p$).
The characteristic ideal of $M^\vee$ is (up to a unit) generated by
\[
f_{M}^{(p)}(T) = f_{M}(T) := p^{\sum_{i} \mu_i} \prod_j f_j(T).
\]
The $\mu$-invariant of $M$ is defined as the power of $p$ in $f_{M}(T)$.
More precisely,
\[
\mu_p(M):=\begin{cases}
\sum_{i=1}^s \mu_i & \textrm{ if } s>0\\
0 & \textrm{ if } s=0.
\end{cases}
\]
The $\lambda$-invariant of $M$ is the degree of the characteristic element, i.e.,
\[
\lambda_p(M) :=\begin{cases}
\sum_{i=1}^s \deg f_i & \textrm{ if } s>0\\
0 & \textrm{ if } s=0.
\end{cases}
\]
The minimal number of generators $g_p(M)$ is defined to be \[g_p(M):=\dim_{\F_p} \left(M^{\vee}\otimes_{\Lambda} \Lambda/\mathfrak{m}\right),\] where $\mathfrak{m}\subset \Lambda$ is its maximal ideal.
Denote by $\mup$ (resp. $\lap$) the $\mu$-invariant (resp. $\lambda$-invariant) of the Selmer group $\Selp$.
The quantity $\gp$ will denote $g_p(\Selp)$.

\begin{lemma}
\label{mupluslambdalemma}
The following inequality holds
\[\lap+\mup\geq \gp.\]
\end{lemma}
\begin{proof}
It follows from \cite[Proposition 4.15]{greenberg1999iwasawa} that $\Selp^{\vee}$ contains no non-zero finite $\Lambda$-submodules.
The result then follows from \cite[Lemma 2.2]{matsuno2007construction}.
\end{proof}

\begin{lemma}
\label{lambdaandrank}
Let $E$ be an elliptic curve defined over $\Q$, then $\lap\geq \op{rank}E(\Q)$.
\end{lemma}
\begin{proof}
We denote by $r$ the $\Z_p$-corank of $\Selp^{\Gamma}$.
Since $\Selp$ is a cotorsion $\Lambda$-module, it follows that $r$ is finite.
This is indeed a simple consequence of the structure theorem for $\Lambda$-modules.
From the short exact sequence, 
\begin{equation}
\label{sesSelmer}
0\rightarrow E(\Q)\otimes \Q_p/\Z_p\rightarrow \Sel_{p^{\infty}}(E/\Q)\rightarrow \Sh(E/\Q)[p^{\infty}]\rightarrow 0,
\end{equation}
it follows that
\begin{equation}
\label{eq21}
\op{corank}_{\Z_p} \Sel_{p^{\infty}}(E/\Q)\geq \rank E(\Q).
\end{equation}
It is an exercise in the structure theory of $\Lambda$-modules that $\lap\geq r$.
Hence, it suffices to show that $r\geq \rank E(\Q)$.
This is indeed the case, since there is a natural map 
\[
\Sel_{p^{\infty}}(E/\Q)\rightarrow \Selp^\Gamma
\]
with finite kernel.
\end{proof}

\section{The Truncated Euler Characteristic}
\label{truncatedEC}
Here, we introduce the notion of the Euler characteristic.
The average size of this Iwasawa theoretic invariant will be studied in due course.
Even though the following discussion holds for any cofinitely generated cotorsion $\Lambda$-module $M$, we will specialize to the case that $M = \Sel_{p^\infty}(E/\Q_\infty)$ where $p$ is a prime of good ordinary reduction.
Since the $p$-cohomological dimension of $\Gamma$ is $1$, we have that $H^i(\Gamma, M)$ is always zero for $i\geq 2$.

\begin{lemma}\label{balancedrank}
For a cofinitely generated cotorsion $\Lambda$-module $M$,
\[\corank_{\Z_p} M^{\Gamma}=\corank_{\Z_p} M_{\Gamma}.\]
\end{lemma}
\begin{proof}
See \cite[Lemma 2.1]{KR21}.
\end{proof}
When the cohomology groups $H^0(\Gamma, M)$ and $H^1(\Gamma, M)$ are finite, the (classical) \emph{Euler characteristic} $\chi(\Gamma, M)$ is defined as follows
\[\chi(\Gamma, M)=\frac{\left(\# H^0(\Gamma, M)\right)}{\left(\# H^1(\Gamma, M)\right)}.\]
For ease of notation, set $\chi(\Gamma, E[p^{\infty}]):= \chi(\Gamma, \Sel_{p^\infty}(E/\Q_\infty))$.
The next lemma gives a criterion for the classical Euler characteristic to be well-defined.
\begin{lemma}
\label{lemma EC defined}
Let $E_{/\Q}$ be an elliptic curve with good ordinary reduction at $p>2$.
The following are equivalent.
\begin{enumerate}
 \item The classical Euler characteristic $\chi(\Gamma, E[p^{\infty}])$ is well-defined.
 \item $\Sel_{p^\infty}(E/\Q_\infty)^{\Gamma}$ is finite.
 \item The Selmer group $\Sel_{p^{\infty}}(E/\Q)$ is finite.
 \item The Mordell--Weil group $E(\Q)$ is finite, i.e., the Mordell--Weil rank is 0.
\end{enumerate} 
\end{lemma}

\begin{proof}
This is proven in \cite[Lemma 3.2]{KR21}. 
We include a proof for the sake of convenience.
Recall that the Selmer group $\Sel_{p^\infty}(E/\Q_\infty)$ is a cotorsion $\Lambda$-module.
By Lemma $\ref{balancedrank}$, $\Sel_{p^\infty}(E/\Q_\infty)^{\Gamma}$ is finite if and only if $\Sel_{p^\infty}(E/\Q_\infty)_{\Gamma}$ is finite.
This shows that $(1)$ and $(2)$ are equivalent.

Mazur's Control Theorem asserts that there is a natural map
\[
\Sel_{p^\infty}(E/\Q)\rightarrow \Sel_{p^\infty}(E/\Q_\infty)^{\Gamma},
\]
with finite kernel and cokernel.
Hence, the conditions $(2)$ and $(3)$ are equivalent.
Recall that we have assumed $\Sh(E/\Q)$ is finite.
Thus, it follows from \eqref{sesSelmer} that conditions $(3)$ and $(4)$ are equivalent.
\end{proof}

When the cohomology groups $H^i(\Gamma, M)$ are \emph{not} finite, there is a generalization of the above notion.
Since $H^1(\Gamma, M)$ is isomorphic to the group of coinvariants $H_0(\Gamma, M)=M_{\Gamma}$, there is a natural map 
\[
\Phi_{M}:M^{\Gamma}\rightarrow M_{\Gamma}
\]
sending $x\in M^{\Gamma}$ to the residue class of $x$ in $M_{\Gamma}$.
It follows from Lemma \ref{balancedrank} that the kernel of $\Phi_M$ is finite if and only if its cokernel is finite.
When the kernel and cokernel of $\Phi_{M}$ are finite, we define the \emph{truncated Euler characteristic}, denoted by $\chi_t(\Gamma, M)$, as the following quotient,
\[
\chi_t(\Gamma, M):=\frac{\#\op{ker}(\Phi_{M})}{\#\op{cok}(\Phi_{M})}.
\]
It is easy to check that when $\chi(\Gamma, M)$ is defined, so is $\chi_t(\Gamma, M)$.
In fact,
\[\chi_t(\Gamma, M)=\chi(\Gamma, M).\]
Express the characteristic element $f_{M}^{(p)}(T)$ as a polynomial, 
\[
f_{M}^{(p)}(T)=c_0+c_1T+\dots +c_d T^d.
\]
Let $r_{M}$ denote the order of vanishing of $f_{M}^{(p)}(T)$ at $T=0$.
For $a,b\in \Q_p$, we write $a\sim b$ if there is a unit $u\in \Z_p^{\times}$ such that $a=bu$.
\begin{lemma}
\label{lemmazerbes}
Let $M$ be a cofinitely generated cotorsion $\Lambda$-module.
Assume that the kernel and cokernel of $\Phi_{M}$ are finite.
Then,
\begin{enumerate}
\item ${r_{M}}=\corank_{\Z_p}(M^{\Gamma})=\corank_{\Z_p}(M_{\Gamma})$.
\item $c_{r_{M}}\neq 0$.
\item $c_{r_{M}}\sim \chi_t(\Gamma, M)$.
\end{enumerate}
\end{lemma}
\begin{proof}
See \cite[Lemma 2.11]{zerbes09}.
\end{proof}

It follows from Lemma $\ref{lemmazerbes}$ that the truncated Euler characteristic, when defined, is always an integer.
Note that $\chi(\Gamma, M)$ is defined if and only if $r_{M}=0$.
When this happens, the constant coefficient $c_0\sim \chi(\Gamma, M)$.
The following lemma indicates the precise relationship between truncated Euler characteristic and the Iwasawa invariants.
\begin{lemma}
\label{TECmulambda}
Let $M$ be a cofinitely generated and cotorsion $\Lambda$-module such that $\Phi_{M}$ has finite kernel and cokernel.
Let $r_{M}$ be the order of vanishing of $f_{M}^{(p)}(T)$ at $T=0$.
Then, the following are equivalent.
\begin{enumerate}
\item $\chi_t(\Gamma, M)=1$,
\item $\mu(M)=0$ and $\lambda(M)=r_{M}$.
\end{enumerate}
\end{lemma}

\begin{proof}
See \cite[Lemma 3.4]{KR21}
\end{proof}

The next result provides conditions for the truncated Euler characteristic to be defined.

\begin{lemma}
\label{truncdefined}
Let $M$ be a $p$-primary, discrete $\Lambda$-module and let $X$ be its Pontryagin dual.
Suppose that $X$ is a finitely generated torsion $\Lambda$-module.
Denote by $X[p^{\infty}]$ the $p^{\infty}$-torsion submodule of $X$.
Let $f_1(T), \dots, f_n(T)$ be distinguished polynomials such that $X/X[p^{\infty}]$ is pseudo-isomorphic to $\bigoplus_{i=1}^n \Lambda/(f_i(T))$.
If $T^2 \nmid f_i(T)$ for any $i$, then the kernel and cokernel of $\Phi_M$ are finite and the truncated Euler characteristic $\chi_t(\Gamma, M)$ is defined.
In particular, $\chi_t(\Gamma, M)$ is defined when $r_M\leq 1$.
\end{lemma}
\begin{proof}
The assertion of the lemma follows from the proof of \cite[Lemma 2.11]{zerbes09}.
\end{proof}

When $E$ has good ordinary reduction at $p$, there is a $p$-adic analog of the usual height pairing, which was studied extensively by P.~Schneider in \cite{Schneider82, Schneider85}.
This $p$-adic height pairing is conjecturally non-degenerate, and its determinant is called the \emph{$p$-adic regulator} (denoted by $\Reg_p(E/\Q)$).
In Iwasawa theory, it is standard to use the following normalized $p$-adic regulator, which is well-defined up to a $p$-adic unit
\[
\mathcal{R}_p(E/\Q) = \frac{\Reg_p(E/\Q)}{p^{\rank E(\Q)}}.
\]
The following result gives the formula for the (truncated) Euler characteristic of the $p$-primary Selmer group (when it is defined).
In the CM case, this was proven by B.~Perrin-Riou (see \cite{PR82}) and in the general case by P.~Schneider (see \cite[Theorem 2']{Schneider85}).

\begin{theorem}
\label{pbsdconj}
Let $E_{/\Q}$ be an elliptic curve with good ordinary reduction at $p>2$.
Then, the order of vanishing of the characteristic element $f_E^{(p)}(T)$ at $T=0$ is at least equal to $\rank \ E(\Q)$.
Now, suppose that the following two conditions hold
\begin{enumerate}[(i)]
 \item $\mathcal{R}_p(E/\Q)\neq 0$,
 \item $\Sh(E/\Q)[p^{\infty}]$ is finite.
\end{enumerate}
Then,
$\op{ord}_{T=0} f_E^{(p)}(T)=\op{rank}E(\Q)$.

Further, if the truncated Euler characteristic $\chi_t(\Gamma, E[p^{\infty}])$ is defined, then it is given by the following $p$-adic Birch and Swinnerton-Dyer formula
\begin{equation}\label{BSDformula}
\chi_{t}(\Gamma, E[p^\infty]) \sim \mathcal{R}_p(E/\Q) \times \frac{\# \Sh(E/\Q)[p^\infty] \times \prod_{\ell\neq p} c_{\ell}^{(p)} \times \left(\# \widetilde{E}(\FF_p)[p^\infty]\right)^2}{\left(\# E(\Q)[p^\infty]\right)^2}.
\end{equation}
Here, $c_{\ell}$ is the Tamagawa number at $\ell\neq p$ and $c_{\ell}^{(p)}$ is its $p$-part.
\end{theorem}

\begin{corollary}
\label{cor37}
Let $E$ be an elliptic curve with good ordinary reduction at an odd prime $p$. 
Further assume that 
\begin{enumerate}
 \item $\op{rank}E(\Q)\leq 1$,
 \item the $p$-adic regulator $\mathcal{R}_p(E/\Q)$ is non-zero, 
 \item $\Sh(E/\Q)[p^{\infty}]$ is finite.
\end{enumerate} Then, the truncated Euler characteristic $\chi_t(\Gamma, E[p^{\infty}])$ is defined and given by \eqref{BSDformula}.
\end{corollary}
\begin{proof}
By the first assertion of Theorem \ref{pbsdconj}, we know that the order of vanishing of $f_E^{(p)}(T)$ at $T=0$ is $\leq 1$.
Hence, by (the last assertion of) Lemma \ref{truncdefined}, the truncated Euler characteristic is well-defined.
Therefore, in view of the additional assumptions made (in this corollary) an application of Theorem \ref{pbsdconj} implies the truncated Euler characteristic up to a $p$-adic unit is given by \eqref{BSDformula}.
\end{proof}

\section{Densities for Weierstrass Models of Elliptic curves}
\label{densities for W Models}

The results in this section are essentially due to Cremona and Sadek \cite{cremona2020local}, except that we work with short Weierstrass equations and a different notion of height.
They are applied in the subsequent analysis of Iwasawa invariants on average.
Any curve $E_{/\Q}$ admits a unique global Weierstrass equation $Y^2=X^3+aX+b$, for integers $a$ and $b$ such that $\gcd(a^3, b^2)$ is not divisible by any $12$-th power.
Set $E_{(a,b)}$ to be the elliptic curve with minimal short Weierstrass equation above.
We introduce some terminology, which will be in place throughout this paper.
The height of $E$ will denote the \emph{na{\"i}ve height}, namely 
\[
H(E_{(a,b)}):=\max\left\{4|a|^3, 27b^2\right\}.
\]

\subsection{}
For any integral domain $R$, denote by $\mathcal{W}(R)$ the set of pairs
\[
\mathcal{W}(R):=\{(a,b)| a,b\in R\}
\]
where to each pair $(a,b)$, we associate the curve 
\[
E_{(a,b)}:Y^2=X^3+aX+b
\] over $R$.
The equation $E_{(a,b)}$ defines an elliptic curve if $\Delta_{(a,b)}:=4a^3+27b^2$ is non-zero.
Let $\ell$ be a prime number and $v$ be the valuation on $\Z_\ell$, normalized by $v(\ell)=1$.
The Haar measure on $\Z_\ell$ induces a measure $\mu$ on $\mathcal{W}(\Z_\ell)$ such that $\mu\left(\mathcal{W}(\Z_\ell)\right)=1$.
For any tuple of non-negative integers $(v_1, v_2)$, consider the set $\mathcal{W}(v_1, v_2)\subseteq \mathcal{W}(\Z_\ell)$ defined by
\[
\mathcal{W}(v_1, v_2):=\left\{(a,b)\in \mathcal{W}(\Z_\ell)\ | \ v(a)\geq v_1, v(b)\geq v_2\right\}.
\]
The measure $\mu$ has the property that $\mu\left(\mathcal{W}(v_1, v_2)\right)=\ell^{-v_1-v_2}$.

Let $\T$ denote a Kodaira type at $\ell$, the choices are as $\T=\I_0, \I_n, \I_n^*$.
Note that after a finite base-change of $\Q_\ell$, one of these types are attained.
Here $\I_0$ is the Kodaira type for good reduction.
For each Kodaira type $\T$, let $\mathcal{W}_{\T}(\Z_\ell)$ be the subset of $\mathcal{W}(\Z_\ell)$ consisting of $(a,b)$ such that $E_{(a,b)}$ has Kodaira type $\T$ at $\ell$.
Denote by $\mathcal{W}_{\op{I}_{\geq n}}(\Z_\ell)$ the union
\[\mathcal{W}_{\op{I}_{\geq n}}(\Z_\ell):=\bigcup_{m\geq n} \mathcal{W}_{\op{I}_m}(\Z_\ell).\]
The local density of elliptic curves of type $\T$ is the measure 
\[
\rho_T(\ell):=\mu\left(\mathcal{W}_{\T}(\Z_\ell)\right).
\]
The pair $(a,b)$ is minimal if the associated Weierstrass equation is minimal, i.e., if either $v(a)<4$ or $v(b)<6$.
Let $\mathcal{W}_M(\Z_\ell)\subset \mathcal{W}(\Z_\ell)$ consist of models which are minimal in this sense.
A global model over $\Z$ is \emph{minimal} if and only if it localizes to the subset $\mathcal{W}_M(\Z_\ell)$ at every prime $\ell$.
Let $\rho^M(\ell)$ be the local density of the Weierstrass equations that are minimal, i.e.,
\[
\rho^M(\ell):=\mu\left(\mathcal{W}_M(\Z_\ell)\right).
\]
We write $\rho^M_{\T}(\ell)$ for the local density of minimal Weierstrass equations of type $\T$, i.e.,
\[
\rho_{\T}^M(\ell):=\mu\left(\mathcal{W}_M^T(\Z_\ell)\right)=\mu\left(\mathcal{W}^{\T}(\Z_\ell)\cap \mathcal{W}_M(\Z_\ell)\right).
\]
Computation of local densities will be related to global densities, and we begin by computing $\rho^M(\ell)$.

\begin{proposition}
\label{prop41}
Let $\ell$ be any prime and $\rho^M(\ell)$ be defined as above.
Then 
\[\rho^M(\ell)=1-\ell^{-10}.\]
\end{proposition}

\begin{proof}
Note that $(a,b)\in \mathcal{W}_M(\Z_\ell)$ if it does not reduce to $(0,0)\in \Z/\ell^4\times \Z/\ell^6$.
Hence, out of a total of $\ell^{10}$ residue classes, there are $\ell^{10}-1$ which lift to $\mathcal{W}_M(\Z_\ell)$.
Therefore, 
\[\rho^M(\ell)=\frac{\ell^{10}-1}{\ell^{10}}=1-\ell^{-10}.\]
\end{proof}
Next, we compute the density corresponding to the Kodaira types $\T=\I_0$ (i.e., good reduction) and $\I_n$ for $n\geq 1$.
We begin by proving the following lemma.

\begin{lemma}
\label{ktypeIn}
Let $\ell\geq 5$ be a prime, $n\geq 1$ and $(a,b)\in \mathcal{W}_M(\Z_\ell)$ with $\Delta_{(a,b)}\neq 0$.
Then, the elliptic curve $E_{(a,b)}$ has Kodaira type $\I_n$ if and only if
\begin{enumerate}
 \item $(a,b)\not \equiv (0,0)\mod{\ell}$,
 \item $v(\Delta_{(a,b)})=n$.
\end{enumerate}
\end{lemma}
\begin{proof}
Throughout this proof, $K$ will denote a finite extension of $\Q_\ell$ in which the Kodaira type is realized and $\pi$ is the uniformizer of $\mathcal{O}_K$.
We show that the result follows from Tate's algorithm (see \cite[\S 0-Summary]{Tat75}).
According to this algorithm, we first require that $v(\Delta_{(a,b)})>0$, i.e., the reduced curve $\tilde{E}$ has a singular point.
We make a change of variables so that the singular point of $\tilde{E}$ is $(0,0)$.
If ${(a,b)}\equiv (0,0)\mod{\ell}$, then this condition is already satisfied for $\tilde{E}$.
In this case, $\pi|b_2$ (see \cite[p. 364]{silverman1994advanced} for the definition) and one can conclude from the algorithm that the Kodaira type is not $\I_n$.
Hence, it is required that ${(a,b)}\not \equiv (0,0)\mod{\ell}$.
It suffices to observe that if additionally, $v(\Delta_{(a,b)})=n$, then the Kodaira type is $\I_n$.
Upon extending the field $\Q_{\ell}$ by a finite extension, it is possible to make $(0,0)$ the singular point after a transformation replacing $x$ by $x+c$, where $c$ is an $\ell$-adic unit.
For the new equation $\pi\nmid b_2$, and as a result, the Kodaira type is $\I_n$, according to the discussion in \textit{loc. cit.}
\end{proof}

\begin{proposition}\label{prop43}
Let $\ell\geq 5$ be a prime and $\rho_{\T}^M(\ell)$ be as defined previously.
Then, for $\T=\I_n$, the following equalities hold.
\begin{enumerate}
 \item $\rho_{\I_0}^M(\ell)=1-\ell^{-1}$.
 \item 
 $\rho_{\I_n}^M(\ell)=\frac{(\ell-1)^2}{\ell^{n+2}}$ for $n\geq 1$.
 \item 
 $\rho_{\I_{\geq n}}^M(\ell)=\frac{(\ell-1)}{\ell^{n+1}}$.
\end{enumerate}
\end{proposition}
\begin{proof}
\noindent For $(1)$ it suffices to show that the number of solutions $(a, b)\in \Z/\ell\times \Z/\ell$ to the equation $\Delta_{(a,b)}=0$ is $\ell$.
When $\ell\neq 2,3$, we see that $27b^2=-4a^3$ if and only if $-3a$ is a quadratic residue modulo $\ell$.
As $a$ ranges over $\Z/\ell\Z$, we find that $-3a$ also ranges over $\Z/\ell\Z$.
If $a=0$, there is one corresponding solution, namely, $(0,0)$ and if $-3a$ is a quadratic residue, then there are exactly $2$-solutions.
Note that there are an equal number of quadratic residues as there are non-residues.
When $\ell=2$ (resp. $\ell=3$), any solution is of the form $(a,0)$ (resp. $(0,b)$).
Hence, in all, there are a total of $\ell$-solutions, and this completes the proof of part (1).

Lemma \ref{ktypeIn} asserts that $\mathcal{W}_{\I_n}(\Z_\ell)$ consists of all minimal pairs $(a,b)\not \equiv (0,0)\mod{\ell}$ for which $v(\Delta_{(a,b)})=n$.
This is determined by the residue of $(a,b)$ in $\Z/\ell^{n+1}\Z\times \Z/\ell^{n+1}\Z$.
In all, there are $\ell-1$ solutions of $\Delta_{(a,b)}=0$ for $(0,0)\neq (\bar{a}, \bar{b})\in\Z/\ell\Z\times \Z/\ell\Z$.
For any integer $m\geq 1$, each solution lifts to $\ell^{m-1}$ solutions of $\Delta_{(a,b)}=0$ in $\Z/\ell^m\Z\times \Z/\ell^m\Z$.
Hence, the number of pairs in $\Z/\ell^{n+1}\Z\times \Z/\ell^{n+1}\Z$ is equal to $\ell^2\left(\ell^{n-1}(\ell-1)\right)-\ell^n(\ell-1)=\ell^n(\ell-1)^2$.
Therefore, we have that 
\[\rho_{\I_n}^M(\ell)=\frac{\ell^n(\ell-1)^2}{\ell^{2(n+1)}}=\frac{(\ell-1)^2}{\ell^{n+2}}.\]
In the same way
\[\rho_{\I_{\geq n}}^M(\ell) 
= \frac{\ell^{n-1}(\ell-1)}{\ell^{2n}}=
\frac{(\ell-1)}{\ell^{n+1}}.\]
\end{proof}

\subsection{}
We now discuss global Weierstrass equations satisfying local constraints.
Let $S$ be a collection of prime numbers (possibly all prime numbers).
For each prime $\ell\in S$, let $\mathcal{U}_\ell$ be a subset of $\mathcal{W}(\Z_\ell)$ which is defined by a congruence condition.
In other words, for some integer $N>0$ (depending on $\ell$), $\mathcal{U}_\ell$ is the set that lifts a certain collection of residue classes $\{x_1, \dots, x_t\}$ in $\mathcal{W}(\Z/\ell^N \Z)$.
Then, the measure $\mu\left(\mathcal{U}_\ell\right)$ is equal to $t/\ell^{2N}$.
Let $\Phi$ be the data consisting of the set of primes $S$ and the choice of $\mathcal{U}_{\ell}$ for each prime $\ell\in S$.
We refer to $\Phi$ as a \emph{local datum over $S$ defined by congruence conditions.}
Let \[
\iota_\ell:\mathcal{W}(\Z)\rightarrow \mathcal{W}(\Z_\ell)
\] denote the natural inclusion map.
Let $\mathcal{W}_M(\Z)$ be the set of minimal Weierstrass models, 
\[
\mathcal{W}_M(\Z):=\left\{ (a,b)\in \mathcal{W}(\Z)\mid \iota_\ell\left((a,b)\right)\in \mathcal{W}_M(\Z_\ell)\text{ for all }\ell\right\}.
\]
\begin{definition}
Let $\Phi=(S, \{\mathcal{U}_\ell\}_{\ell\in S})$ be as above, and let $\mathcal{W}_\Phi$ consist of the set of $(a,b)\in \Z\times \Z$ such that $\iota_\ell\left((a,b)\right)\in \mathcal{U}_\ell$ for each prime $\ell\in S$.
For $x>0$, let $\mathcal{W}_\Phi(x)$ be the subset of $\mathcal{W}_\Phi$ consisting of pairs $(a,b)$ with $H(E_{(a,b)})\leq x$.
The density of $\mathcal{W}_{\Phi}$ is defined with respect to the na{\"i}ve height, as follows
\[\mathfrak{d}_{\Phi}:=\lim_{x\rightarrow \infty} \frac{\#\mathcal{W}_\Phi(x)}{\#\mathcal{W}(x)}.\]
\end{definition}
If we need to impose finitely many local conditions then we have the following result.

\begin{theorem}
\label{th45}
Let $S$ be a collection of finitely many primes and $\Phi=(S, \{\mathcal{U}_\ell\}_{\ell\in S})$ be a local datum over $S$ defined by congruence conditions.
Then, the density $\mathfrak{d}_\Phi$ is as follows
\[\mathfrak{d}_\Phi=\prod_{\ell\in S} \mu\left(\mathcal{U}_\ell\right),\]where the limit is taken with respect to the ascending order for $S$.
\end{theorem}
\begin{proof}
The proof of this result follows from \cite[Proposition 3.2]{cremona2020local}.
The only difference is in the notion of height used, here we work with the na{\"i}ve height.
Their proof however applies to this case with very minor adaptations.
\end{proof}


\begin{theorem}\label{Theorem on admissible sets}
Let $\{\mathcal{U}_{\ell}\}_{\ell}$ be a family determined by congruences and let $S$ be a finite set of primes $\ell$. Let $\Phi=\{\mathcal{V}_\ell\}_\ell$ be the of conditions over all primes such that
\[\mathcal{V}_\ell:=\begin{cases}
\mathcal{U}_{\ell}&\text{ for }\ell\in S,\\
\mathcal{U}_{\ell}'&\text{ for }\ell\notin S.\\
\end{cases}\]Here, $\mathcal{U}_{\ell}'$ is the complement of $\mathcal{U}_\ell$. Then, we have that
\[
\mathfrak{d}_{\Phi} = \prod_{\ell\in S}\mu(\mathcal{U}_{\ell})\times \prod_{\ell\notin S} \left(1-\mu(\mathcal{U}_\ell)\right).
\]
\end{theorem}

\begin{proof}
Since the sum $\sum_\ell \mu(\mathcal{U}_\ell)$ is assumed to converge, the collection $\{\mathcal{U}_\ell\}$ forms an \emph{admissible family} in the sense of \cite[Definition 4]{cremona2020local}. Since $\mathcal{U}_\ell$ is determined by congruences, it is both open and closed, hence, its boundary is empty. Thus, the result follows from \cite[Proposition 3.4]{cremona2020local}.
\end{proof}

The following is an application of the above result.

\begin{corollary}\label{new corollary}
Let $\{\mathcal{U}_{\ell}\}_{\ell}$ be a family determined by congruences and assume that $\sum_\ell \left(1-\mu(\mathcal{U}_\ell)\right)$ converges. Then, we have that
\[
\mathfrak{d}_{\Phi} = \prod_{\ell\in S}\mu(\mathcal{U}_{\ell}).
\]
\end{corollary}
\begin{proof}
Apply Theorem \ref{Theorem on admissible sets} to $S=\emptyset$ to $\{\mathcal{U}_\ell'\}_\ell$, where $\mathcal{U}_\ell'$ is the complement of $\mathcal{U}_\ell$.
\end{proof}
\begin{corollary}
\label{cor46}
The following assertions are true.
\begin{enumerate}
 \item The density of the set of minimal Weierstrass equations is given by 
\[\rho^M=\prod_{\ell} (1-\ell^{-10})=\frac{1}{\zeta(10)},\] as $\ell$ runs over all prime numbers.
\item Let $\Sigma$ be a finite set of primes $\geq 5$.
Let $\mathcal{W}_{\Sigma,n}$ be the set of minimal Weierstrass equations for which the Kodaira-type of the associated elliptic curve $E_{(a,b)}$ is $\I_n$ for the primes $\ell\in \Sigma$.
Then,
\[
\mathfrak{d}\left(\mathcal{W}_{\Sigma,n}\right)=\prod_{\ell\in \Sigma} \frac{(\ell-1)^2}{\ell^{n+2}}\times \prod_{\ell\notin \Sigma} (1-\ell^{-10}).
\]
\end{enumerate}
\end{corollary}
\begin{proof}
The proof of (1) is immediate from Proposition \ref{prop41} and Theorem \ref{Theorem on admissible sets}, setting $S$ to be the empty set, and $\mathcal{U}_\ell=\mathcal{W}_M(\Z_\ell)$ the set of minimal Weierstrass equations.
Part (2) is an immediate consequence of Theorem \ref{Theorem on admissible sets} along with Propositions \ref{prop41} and \ref{prop43}(2).
\end{proof}

Let $\mathscr{E}$ denote the set of \emph{all} elliptic curves $E_{/\Q}$ ordered by na{\"i}ve height.
For $x>0$ and any subset $\mathscr{E}'$ of $\mathscr{E}$, define
\[\mathscr{E}'(x):=\{E\in \mathscr{E}| H(E)\leq x\}.\]
The upper density of $\mathscr{E}'$ is defined as follows
\[\mathfrak{d}\left(\mathscr{E}'\right):=\limsup_{x\rightarrow \infty} \frac{\mathscr{E}'(x)}{\mathscr{E}(x)}.\]
We have the following result.
\begin{lemma}\label{lemma47}
With the above notation,
\[\lim_{x\rightarrow \infty} \frac{\# \mathscr{E}(x)}{\# \mathcal{W}(x)}=\rho^M=\frac{1}{\zeta(10)}.\]
\end{lemma}
\begin{proof}
Note that $\mathscr{E}$ and $\mathcal{W}_M(\Z)$ differ by the set of global minimal Weierstrass models with $\Delta_{(a,b)}=0$, thus defining singular curves.
The set with $\Delta_{(a,b)}=0$ is of density $0$ in the set of all Weierstrass equations $\mathcal{W}$.
Therefore, by Corollary \ref{cor46},
\[\lim_{x\rightarrow \infty} \frac{\# \mathscr{E}(x)}{\# \mathcal{W}(x)}=\lim_{x\rightarrow \infty} \frac{\# \mathcal{W}_M(x)}{\# \mathcal{W}(x)}=\rho^M=\frac{1}{\zeta(10)}.\]
\end{proof}

\section{The Euler characteristic on Average}
\label{EC on average}
We fix a prime number $p$ and let $E$ vary over all elliptic curves defined over the rationals with good ordinary reduction at $p$, ordered by na{\"i}ve height.
Associated to each such elliptic curve $E_{/\Q}$ is its $p$-adic truncated Euler characteristic, $\chi_t(\Gamma, E[p^{\infty}])$.
In Corollary \ref{cor37} it was shown that (under standard assumptions) this generalization of the usual Euler characteristic is always well-defined if $\rank \ E(\Q)\leq 1$.
When $\rank \ E(\Q)=0$, the usual Euler characteristic $\chi(\Gamma, E[p^{\infty}])$ is well-defined and $\chi_t(\Gamma, E[p^{\infty}])=\chi(\Gamma, E[p^{\infty}])$.

Throughout this section, assume that $\mathcal{R}_p(E/\Q)\neq 0$.
This is a well-known conjecture.
For evidence towards this conjecture, we refer the reader to \cite{CM94, Wut04}.
Moreover, we assume that $\Sh(E/\Q)[p^{\infty}]$ is finite for all elliptic curves $E_{/\Q}$.
Then, by Corollary \ref{cor37} we know that (up to a $p$-adic unit)
\[
\chi_t(\Gamma, E[p^{\infty}])\sim \frac{\mathcal{R}_p(E/\Q)\times \# \Sh(E/\Q)[p^{\infty}]\times \left(\prod_{\ell}c_{\ell}^{(p)}(E)\right)\times \left(\# \widetilde{E}(\F_p)[p^\infty]\right)^2 }{\left(\#E(\Q)[p^{\infty}]\right)^2}.
\]
Henceforth, we use the following shorthand notation \begin{enumerate}
\item $r_p(E):=\abs{\mathcal{R}_p(E/\Q)}_p^{-1}$,
 \item $s_p(E):=\#\Sh(E/\Q)[p^{\infty}]$, 
 \item $\tau_p(E):=\prod_{\ell} c_{\ell}^{(p)}(E)$,
 \item $\alpha_p(E):=\#\left(\widetilde{E}(\F_p)[p]\right)$.
\end{enumerate}
The above formula can be rewritten as
\begin{equation}
\label{ecfshort}
\chi_t(\Gamma, E[p^{\infty}])\sim\frac{r_p(E)\times s_p(E)\times \tau_p(E)\times \alpha_p(E)^2}{\left(\#E(\Q)[p^\infty]\right)^2}.\end{equation}

\subsection{}
We study the following question.
\begin{question}
Given an integer $n>0$, what is the proportion of elliptic curves $E_{/\Q}$ with good ordinary reduction at $p$ for which $\chi_t(\Gamma, E[p^{\infty}])\geq p^n$.
More precisely, let $\mathscr{E}_{p^n| \chi}$ be the set of elliptic curves with good ordinary reduction at $p$ for which \emph{either} one of the following conditions is satisfied
\begin{enumerate}[(i)]
 \item $\rank \ E(\Q)\geq 2$, 
 \item $\rank \ E(\Q)\leq 1$ and $p^n| \chi_t(\Gamma, E[p^{\infty}])$.
\end{enumerate}
Then, what is the density of the set $\mathscr{E}_{p^n| \chi}$?
\end{question}
It is conjectured that the set of elliptic curves $E$ with $\rank \ E(\Q)\geq 2$ is a density zero set.
Therefore, the density of $\Echi$ \emph{should coincide} with the density of elliptic curves $E_{/\Q}$ for which the truncated Euler characteristic $\chi_t(\Gamma, E[p^{\infty}])$ is well-defined and divisible by $p^n$.
In the above formula \eqref{ecfshort}, the quantities $r_p(E)$ and $s_p(E)$ are fine invariants for which it appears to be difficult to prove unconditional results on their average behavior.

First, we look at the quantity $\#E(\Q)[p^{\infty}]$ which appears in the denominator of \eqref{ecfshort}.
In \cite{duke1997elliptic}, W.~Duke has shown that almost all elliptic curves over $\Q$ have trivial torsion.
In particular, $\# E(\Q)[p^{\infty}]=1$ for $100\%$ of elliptic curves.
Thus on average, this term does not contribute to $\chi_t(\Gamma, E[p^{\infty}])$.
To prove lower bounds for the density of $\Echi$, we analyze the quantity 
\[
\xi_p(E):=\tau_p(E)\times \alpha_p(E)^2.
\]
Given a set of primes $\Sigma=\{\ell_1, \dots, \ell_m\}$ not containing $p$, $\ell_i\geq 5$, and an integer $k\geq 1$, we define two sets of elliptic curves.
Let $\mathscr{E}_{\Sigma,k}$ (resp. $\mathscr{E}_{\Sigma,k}'$) consist of $E_{/\Q}\in \mathscr{E}$ such that the following conditions are satisfied
\begin{enumerate}[(i)]
 \item $E$ has good ordinary reduction at $p$,
 \item $p$ divides $c_\ell(E)$ if $\ell\in \Sigma$, and the Kodaira-type of $E$ at $\ell$ is $\I_{np}$ for some integer $n\leq k$,
 \item for $\ell\notin \Sigma\cup \{2,3 , p\}$, the Kodaira- type of $E$ at $\ell$ is \emph{not} of the form $\op{I}_m$ for any $m\geq p$,
 \item $\alpha_p(E)=1$ (resp. $p| \alpha_p(E)$).
\end{enumerate}
It follows from the definition that $\mathscr{E}_{\Sigma,k}$ and $\mathscr{E}_{\Sigma,k}'$ are all mutually disjoint.
Associated with the sets $\mathscr{E}_{\Sigma,k}$ and $\mathscr{E}_{\Sigma,k}'$ are sets of minimal Weierstrass equations, determined by congruence data $\Phi_{\Sigma,k}=\{\mathcal{U}_{\ell, k}\mid \ell=p\text{ or }\ell\geq 5\}$ and $\Phi_{\Sigma,k}'=\{\mathcal{U}_{\ell, k}'\mid \ell=p\text{ or }\ell\geq 5\}$ respectively, which we now define.
For each prime $\ell\notin \Sigma\cup\{2,3,p\}$ define $\mathcal{U}_{\ell,k}=\mathcal{U}_{\ell,k}'$ to be the complement of $\mathcal{W}_{\op{I}_{\geq p}}(\Z_\ell)$ in the minimal condition $\mathcal{W}_M(\Z_\ell)$.
For $\ell\in \Sigma$, set $\mathcal{U}_{\ell,k}=\mathcal{U}_{\ell,k}'=\bigsqcup_{j= 1}^k\mathcal{W}_{\I_{jp}}(\Z_\ell)$.
Let $\mathfrak{S}_p\subset\mathcal{W}(\Z/p\Z)$ consisting of all pairs $(a,b)$ for which 
\begin{enumerate}[(i)]
 \item $\Delta_{(a,b)}\neq 0$,
 \item $\# {E_{(a,b)}}(\F_p)\not \equiv 0,1 \mod{p}$.
\end{enumerate}
The criterion $\# {E_{(a,b)}}\not\equiv 1 \mod p$ ensures that $E_{(a,b)}$ has good ordinary reduction at $p$.
The condition $\# {E_{(a,b)}}(\F_p)\not \equiv 0 \mod p$ is equivalent to saying that $p$ is not an anomalous prime (in the sense of \cite{Maz72}).
Let $\mathfrak{S}_p'\subset\mathcal{W}(\Z/p\Z)$ be the subset for which the following conditions are satisfied 
\begin{enumerate}[(i)]
 \item $\Delta_{(a,b)}\neq 0$,
 \item $\# {E_{(a,b)}}(\F_p)\equiv 0\mod{p}$.
\end{enumerate}
Data for $\mathfrak{S}_p$ and $\mathfrak{S}_p'$ is compiled in Tables \ref{tab:1} and \ref{tab:2}, respectively.
Consider the reduction map,
\[
\pi_p:\mathcal{W}_M(\Z_p)\rightarrow \mathcal{W}(\Z/p\Z)
\]
and define $\mathcal{U}_{p,k}:=\pi_p^{-1}\left(\mathfrak{S}_p\right)$ and $\mathcal{U}_{p,k}':=\pi_p^{-1}\left(\mathfrak{S}_p'\right)$.
The congruence data for the sets $\mathscr{E}_{\Sigma, k}$ and $\mathscr{E}_{\Sigma, k}'$ are defined by $\mathcal{U}_{\ell, k}$ and $\mathcal{U}_{\ell, k}'$, respectively as $\ell$ ranges over all primes.
Set $\mathscr{E}_\Sigma:=\bigcup_{k\geq 1} \mathscr{E}_{\Sigma,k} $ and $\mathscr{E}_\Sigma':=\bigcup_{k\geq 1} \mathscr{E}_{\Sigma,k}'$.

\begin{lemma}\label{lemma52}
Let $\Sigma$ be a finite set of primes not containing $\{2,3,p\}$.
For any integer $k$, 
\[\begin{split}
 &\mathfrak{d}\left(\mathscr{E}_{\Sigma, k}\right)>\frac{1}{\zeta(p)}\times \prod_{\ell\in \Sigma} \left(\sum_{j=1}^k \frac{(\ell-1)^2}{(\ell^{10}-1)\ell^{jp-8}}\right)\times \left(\frac{p^{8}\#\mathfrak{S}_p}{p^{10}-1}\right)\\
 &\mathfrak{d}\left(\mathscr{E}_{\Sigma, k}'\right)>\frac{1}{\zeta(p)}\times\prod_{\ell\in \Sigma} \left(\sum_{j=1}^k \frac{(\ell-1)^2}{(\ell^{10}-1)\ell^{jp-8}}\right)\times \left(\frac{p^{8}\#\mathfrak{S}_p'}{p^{10}-1}\right).
\end{split}\]
In particular,
\[\begin{split}
 &\mathfrak{d}\left(\mathscr{E}_{\Sigma}\right)>\frac{1}{\zeta(p)}\times\prod_{\ell\in \Sigma} \left( \frac{\ell^8(\ell-1)^2}{(\ell^{10}-1)(\ell^p-1)}\right)\times \left(\frac{p^{8}\#\mathfrak{S}_p}{p^{10}-1}\right)\\
 &\mathfrak{d}\left(\mathscr{E}_{\Sigma}'\right)>\frac{1}{\zeta(p)}\times\prod_{\ell\in \Sigma} \left( \frac{\ell^8(\ell-1)^2}{(\ell^{10}-1)(\ell^p-1)}\right)\times \left(\frac{p^{8}\#\mathfrak{S}_p'}{p^{10}-1}\right).
\end{split}\]
\end{lemma}
\begin{proof}
We compute the upper density of $\mathscr{E}_{\Sigma, k}$ and $\mathscr{E}_\Sigma$.
The computation for $\mathscr{E}_{\Sigma, k}'$ and $\mathscr{E}_\Sigma'$ is analogous.
Recall that $\mathfrak{d}\left(\mathscr{E}_{\Sigma, k}\right)$ denotes the limit
\[\mathfrak{d}\left(\mathscr{E}_{\Sigma, k}\right):=\limsup_{x\rightarrow \infty} \frac{\#\mathscr{E}_{\Sigma, k}(x)}{\#\mathscr{E}(x)}.\]
It follows from Lemma \ref{lemma47} that
\begin{equation}\label{eqn1}\mathfrak{d}\left(\mathscr{E}_{\Sigma, k}\right)=\zeta(10)\times \limsup_{x\rightarrow \infty}\frac{\#\mathscr{E}_{\Sigma, k}(x)}{\#\mathcal{W}(x)}.\end{equation}
In this setting, Corollary \ref{new corollary} asserts that
\[
\lim_{x\rightarrow \infty}\frac{\#\mathscr{E}_{\Sigma, k}(x)}{\#\mathcal{W}(x)}=\prod_{\ell} \mu\left(\mathcal{U}_{\ell, k}\right).
\]
Using Propositions \ref{prop41} and \ref{prop43}(3) for the case when $\ell\not\in\Sigma\cup \{2,3,p\}$ and Proposition \ref{prop43}(2) when $\ell\in \Sigma$, we obtain 
\begin{equation}\label{eqn2}
\mu\left(\mathcal{U}_{\ell,k}\right)=\rho^M(\ell)=
\begin{cases}
1-\ell^{-10}-\frac{\ell-1}{\ell^{p+1}} &\textrm{when }\ell\notin \Sigma\cup \{2,3,p\}\\
\sum_{j=1}^k \frac{(\ell-1)^2}{\ell^{jp+2}} & \textrm{when } \ell\in \Sigma.
\end{cases}
\end{equation}
Recall that $\mathcal{U}_{p,k}:=\pi_p^{-1}\left(\mathfrak{S}_p\right)$, and hence, 
\begin{equation}\label{eqn3}\mu\left(\mathcal{U}_{p,k}\right)=\frac{\# \mathfrak{S}_p}{p^2}.\end{equation}
From \eqref{eqn1}, \eqref{eqn2} and \eqref{eqn3}, we obtain the following expression
\[ \begin{split}\mathfrak{d}\left(\mathscr{E}_{\Sigma, k}\right)= & \zeta(10)\times \prod_{\ell\notin \Sigma\cup \{2,3,p\}} \left(1-\ell^{-10}-\frac{\ell-1}{\ell^{p+1}}\right)\times \prod_{\ell \in \Sigma} \left(\sum_{j=1}^k \frac{(\ell-1)^2}{\ell^{jp+2}}\right)\times \left(\frac{\# \mathfrak{S}_p}{p^2}\right) \\
 >& \prod_{\ell\notin \Sigma\cup \{2,3,p\}} \left(\frac{1-\ell^{-10}-\frac{\ell-1}{\ell^{p+1}}}{1-\ell^{-10}}\right)\times \prod_{\ell \in \Sigma} \left(\sum_{j=1}^k \frac{(\ell-1)^2}{(1-\ell^{-10})\ell^{jp+2}}\right)\times \left(\frac{\# \mathfrak{S}_p}{(1-p^{-10})p^2}\right) \\
=&\prod_{\ell\notin \Sigma\cup \{2,3,p\}}\left(1-\frac{(\ell-1)}{(\ell^{10}-1)\ell^{p-9}}\right)\times\prod_{\ell\in \Sigma} \left(\sum_{j=1}^k \frac{(\ell-1)^2}{(\ell^{10}-1)\ell^{jp-8}}\right)\times \left(\frac{p^{8}\#\mathfrak{S}_p}{p^{10}-1}\right).\\\end{split}\]
Now notice that 
\[\prod_{\ell\notin \Sigma\cup \{2,3,p\}}\left(1-\frac{(\ell-1)}{(\ell^{10}-1)\ell^{p-9}}\right)> \prod_{\ell\notin \Sigma\cup \{2,3,p\}}\left(1-\frac{1}{\ell^p}\right)> \frac{1}{\zeta(p)},\] from which we obtain the result for $\mathfrak{d}(\mathscr{E}_{\Sigma, k})$.
In order to deduce the result for $\mathscr{E}_\Sigma$, it suffices to note that for all $k\geq 1$, the set $\mathscr{E}_{\Sigma,k}$ is contained in $\mathscr{E}_\Sigma$, and hence,
\[\begin{split}\mathfrak{d}\left(\mathscr{E}_\Sigma \right) &\geq \lim_{k\rightarrow \infty} \mathfrak{d}\left(\mathscr{E}_{\Sigma,k} \right)\\
& > \frac{1}{\zeta(p)}\times\prod_{\ell\in \Sigma} \left(\sum_{j=1}^\infty \frac{(\ell-1)^2}{(\ell^{10}-1)\ell^{jp-8}}\right)\times \left(\frac{p^{8}\#\mathfrak{S}_p}{p^{10}-1}\right)\\
& = \frac{1}{\zeta(p)}\times\prod_{\ell\in \Sigma} \left( \frac{\ell^8(\ell-1)^2}{(\ell^{10}-1)(\ell^p-1)}\right)\times \left(\frac{p^{8}\#\mathfrak{S}_p}{p^{10}-1}\right).\end{split}\]
\end{proof}
For a given integer $n$, let $\mathcal{S}_n$ be the set of all sets $\Sigma=\{\ell_1, \dots, \ell_n\}$ of primes not containing $\{2,3,p\}$ such that $\# \Sigma=n$.
If $n\leq 0$, the set $\mathcal{S}_n$ is taken to be empty.
The notation $\Sigma\in \mathcal{S}_n$ means that $\Sigma$ is a set in the collection $\mathcal{S}_n$.
The sum over $\mathcal{S}_0$ is taken to be equal to $1$, and that over $\mathcal{S}_n$ for $n<0$ is taken to be $0$.
From our computations, we obtain the following lower bound for the density $\mathscr{E}_{p^n|\chi}$,
\[\mathfrak{d}\left(\mathscr{E}_{p^n|\chi}\right):=\limsup_{x\rightarrow \infty} \frac{\# \mathscr{E}_{p^n|\chi}(x)}{\# \mathscr{E}(x)}.\]
\begin{theorem}
\label{thm: lower bound}
With notation as above,
\[ \begin{split}\mathfrak{d}\left(\mathscr{E}_{p^n|\chi}\right)>& \frac{1}{\zeta(p)}\times \sum_{\Sigma\in \mathcal{S}_n}\left( \prod_{\ell\in \Sigma} \left( \frac{\ell^8(\ell-1)^2}{(\ell^{10}-1)(\ell^p-1)}\right)\right)\times \left(\frac{p^{8}\#\mathfrak{S}_p}{p^{10}-1}\right)\\
+ & \frac{1}{\zeta(p)}\times\sum_{\Sigma\in \mathcal{S}_{n-2}}\left( \prod_{\ell\in \Sigma} \left( \frac{\ell^8(\ell-1)^2}{(\ell^{10}-1)(\ell^p-1)}\right)\right)\times \left(\frac{p^{8}\#\mathfrak{S}_p'}{p^{10}-1}\right).\end{split}\]
\end{theorem}
We remind the reader that data for $\mathfrak{S}_p$ and $\mathfrak{S}_p'$ is recorded in Section \ref{tables}.
\begin{proof}
Recall that $\xi_p(E)=\tau_p(E) \times \alpha_p(E)^2$.
Let $\mathscr{E}'\subset \mathscr{E}$ consist of those elliptic curves for which $p^n | \xi_p(E)$.
It follows from \eqref{ecfshort} that if $E\in \mathscr{E}'$ and $E(\Q)[p]=0$, then $p^n| \chi_t(\Gamma, E[p^{\infty}])$.
Recall that by the aforementioned theorem of Duke, $E(\Q)[p]=0$ for $100\%$ of elliptic curves.
Hence $\mathfrak{d}\left(\mathscr{E}_{p^n| \chi}\right)\geq \mathfrak{d}\left(\mathscr{E}'\right)$.

Let $\mathcal{E}_1$ (resp. $\mathcal{E}_2$) be the set of elliptic curves $E_{/\Q}$ for which $\alpha_p(E)=1$ and $p^n$ divides $\tau_p(E)$ (resp. $p| \alpha_p(E)$ and $p^{n-2}| \tau_p(E)$).
The two sets are disjoint and
\[
\mathscr{E}' \supseteq \mathcal{E}_1 \cup \mathcal{E}_2.\]
For $\Sigma\in \mathcal{S}_n$, the set $\mathscr{E}_{\Sigma}$ is contained in $\mathcal{E}_1$.
On the other hand, for $\Sigma\in \mathcal{S}_{n-2}$, the set $\mathscr{E}_\Sigma'$ is contained in $\mathcal{E}_2$.
Note that, if $n\leq 2$, the statement is vacuous since $\mathcal{S}_{n-2}$ is the empty set.
Since the sets $\mathscr{E}_\Sigma$ and $\mathscr{E}_\Sigma'$ are all mutually disjoint, we obtain a lower bound
\[\mathfrak{d}\left(\mathscr{E}_{p^n|\chi}\right)\geq \sum_{\Sigma\in \mathcal{S}_{n}} \mathfrak{d}\left(\mathscr{E}_\Sigma\right)+ \sum_{\Sigma\in \mathcal{S}_{n-2}} \mathfrak{d}\left(\mathscr{E}_\Sigma'\right).\] The result now follows from Lemma \ref{lemma52}.
\end{proof}

\section{Iwasawa Invariants on Average}
\label{Invariants on Average}
\par In this section, $p$ will be a fixed odd prime number and $E$ will vary over a certain set of elliptic curves.
For any positive integer $n$, we show that the set of elliptic curves $E$ for which $\mup+\lap\geq n$ has positive density, and further, there is an explicit lower bound for this density expressible in terms of $p$ and $n$ alone.
It is a conjecture of Greenberg (see \cite[Conjecture 1.11]{greenberg1999iwasawa}) that when $E$ has good ordinary reduction at $p$ and the residual representation $E[p]$ is irreducible as a module over $\Gal(\bar{\Q}/\Q)$, then $\mup=0$.
Therefore, one expects that $\mup=0$ for $100\%$ of the elliptic curves.
Assuming Greenberg's conjecture, the results in this section imply that $\lap\geq n$ for a certain explicitly defined proportion of elliptic curves.
We now make these notions more precise.
For an elliptic curve $E_{/\Q}$ with good ordinary reduction at $p$, recall that $\gp$ denotes the minimum number of generators for $\Selp^{\vee}$.
Recall that Lemma \ref{mupluslambdalemma} asserts that \begin{equation}\label{mupluslambda}\mup+\lap\geq \gp.\end{equation}

Let $n$ be a positive integer and set $\mathscr{E}_{g\geq n}$ (resp. $\mathscr{E}_{\mu+\lambda\geq n}$) to be the set of elliptic curves $E_{/\Q}$ ordered by na{\"i}ve height for which the following conditions are satisfied
\begin{enumerate}[(i)]
 \item $E$ has good ordinary reduction at $p$,
 \item $\gp\geq n$ (resp. $\mup+\lap\geq n$).
\end{enumerate}
Here, we analyze lower bounds for the density $\mathfrak{d}\left(\mathscr{E}_{g\geq n}\right)$.
It follows from \eqref{mupluslambda} that
\[\mathfrak{d}\left(\mathscr{E}_{\mu+\lambda\geq n}\right)\geq \mathfrak{d}\left(\mathscr{E}_{g\geq n}\right).\]
We prove an explicit lower bound for $\mathfrak{d}\left(\mathscr{E}_{g\geq n}\right)$, thus obtaining one for $\mathfrak{d}\left(\mathscr{E}_{\mu+\lambda\geq n}\right)$ as well.
\subsection{} \par We first prove an algebraic criterion for the invariant $g$ to be larger than a given number $n$.
As always, let $E_{/\Q}$ be an elliptic curve with good ordinary reduction at a fixed odd prime $p$.
Let $E(\Q)^*$ be the inverse limit $E(\Q)^*:=\varprojlim_n E(\Q)/p^n$, and set 
\[
\widehat{E(\Q)^*}:=\Hom\left(E(\Q)^*, \Q_p/\Z_p\right)
\]
its Pontryagin dual.
\begin{lemma}
\label{cokernelPhi}
Let $E_{/\Q}$ be an elliptic curve and $p$ be a prime for which $\Sh(E/\Q)[p^{\infty}]$ is finite.
Then, the map $\Phi_{E,\Q}$ fits into an exact sequence
\[
H^1\left(\Q_S/\Q,E[p^{\infty}]\right)\xrightarrow{\Phi_{E,\Q}} \bigoplus_{\ell\in S} J_\ell(E/\Q)\rightarrow \widehat{E(\Q)^*}.
\]
\end{lemma}
\begin{proof}
The compact Selmer group $\mathfrak{S}_{p^{\infty}}(E/\Q)$ is defined as the inverse limit
\[
\mathfrak{S}_{p^{\infty}}(E/\Q):=\varprojlim_n \Sel_{p^n}(E/\Q).
\]
By the Cassels-Poitou-Tate exact sequence, we know that the map $\Phi_{E, \Q}$ fits into a long exact sequence (see for example, \cite[p. 9]{coates2000galois})
\[0\rightarrow \Sel_{p^{\infty}}(E/\Q)\rightarrow H^1\left(\Q_S/\Q,E[p^{\infty}]\right)\xrightarrow{\Phi_{E,\Q}} \bigoplus_{\ell\in S} J_\ell(E/\Q)\rightarrow \widehat{\mathfrak{S}_{p^{\infty}}(E/\Q)}\rightarrow \dots.\]
When $\Sh(E/\Q)[p^{\infty}]$ is finite, it follows that $\mathfrak{S}_{p^{\infty}}(E/\Q)$ is identified with $E(\Q)^*$ (see \cite[proof of Proposition 1.9]{coates2000galois}).
\end{proof}

\begin{definition}
Let $E_{/\Q}$ be an elliptic curve with good ordinary reduction at the prime $p$.
Define $\mathfrak{a}_E^{(p)}$ to be the number of primes $\ell\neq p$ for which $p$ divides the Tamagawa number $c_\ell(E)$.
Set 
\[
\mathfrak{b}_E^{(p)}:=
\begin{cases}
1 & \textrm{when } p|\# \widetilde{E}(\F_p)\\
0 & \textrm{otherwise.}
\end{cases}
\]
Finally, write $\mathfrak{c}_E^{(p)}:=\mathfrak{a}_E^{(p)}+\mathfrak{b}_E^{(p)}$.
\end{definition}We have the following characterization of the minimal number of generators $\gp$.
\begin{lemma}
\label{lambda+mu in terms of tamagawa and anomalous}
Let $E_{/\Q}$ be an elliptic curve for which the following conditions are satisfied.
\begin{enumerate}
 \item $E$ has good ordinary reduction at $p$,
 \item $E(\Q)[p]=0$,
 \item $\Sh(E/\Q)[p^{\infty}]$ is finite.
\end{enumerate}
Then, $\gp\geq \mathfrak{c}_E^{(p)}$.
In particular,
\begin{equation}\lap+\mup\geq \mathfrak{c}_E^{(p)}.
\end{equation}
\end{lemma}

\begin{proof}
Since it is assumed that $E(\Q)[p]=0$, it follows that $\widehat{E(\Q)^*}\simeq (\Q_p/\Z_p)^{\rank \ E(\Q)}$.
Lemma \ref{cokernelPhi} asserts that there is an exact sequence
\[H^1\left(\Q_S/\Q,E[p^{\infty}]\right)\xrightarrow{\Phi_{E,\Q}} \bigoplus_{\ell\in S} J_\ell(E/\Q)\rightarrow \widehat{E(\Q)^*}.\]
Consider the diagram 
\begin{equation}\label{diagram1}
\begin{tikzcd}[column sep = small, row sep = large]
0\arrow{r} & \Sel_{p^{\infty}}(E/\Q) \arrow{r} \arrow{d}{f} & H^1\left(\Q_S/\Q,E[p^{\infty}]\right) \arrow{r} \arrow{d}{g} & \op{im} \Phi_{E,\Q} \arrow{r} \arrow{d}{h'} & 0\\
0\arrow{r} & \Sel_{p^{\infty}}(E/\Q_{\infty})^{\Gamma} \arrow{r} & H^1\left(\Q_S/\Q,E[p^{\infty}]\right)^{\Gamma} \arrow{r} &\left(\bigoplus_{\ell\in S} J_\ell(E/\Q_{\infty})\right)^\Gamma.
\end{tikzcd}\end{equation}
Since $E(\Q)[p]=0$, it follows from standard arguments that the map $g$ is an isomorphism (see \cite[p. 34]{coates2000galois}).
As a result, the snake lemma applied to \eqref{diagram1} implies that there is a short exact sequence
\begin{equation}\label{es1}
 0\rightarrow \Sel_{p^{\infty}}(E/\Q)\rightarrow \Sel_{p^{\infty}}(E/\Q_{\infty})^{\Gamma}\rightarrow \ker h'\rightarrow 0.
\end{equation}
Next we have the diagram
\begin{equation}\label{diagram2}
\begin{tikzcd}[column sep = small, row sep = large]
0\arrow{r} & \op{im} \Phi_{E,\Q} \arrow{r} \arrow{d}{h'} & \bigoplus_{\ell\in S} J_\ell(E/\Q) \arrow{r} \arrow{d}{h} & \widehat{E(\Q)^*} \arrow{d}{j} \\
0\arrow{r} & \left(\bigoplus_{\ell\in S} J_\ell(E/\Q_{\infty})\right)^\Gamma \arrow{r} & \left(\bigoplus_{\ell\in S} J_\ell(E/\Q_{\infty})\right)^\Gamma \arrow{r} & 0.
\end{tikzcd}\end{equation}
This yields
\begin{equation}\label{es2}
 0\rightarrow \ker h'\rightarrow \ker h\rightarrow \widehat{E(\Q)^*}.
\end{equation}
Using the sequences \eqref{es1} and \eqref{es2}, we estimate 
\[
\gp=\dim_{\F_p} \Sel_{p^{\infty}}(E/\Q_{\infty})^{\Gamma}[p].
\]
From the short exact sequence \eqref{es1}, we obtain the following exact sequence
\[0\rightarrow \Sel_{p^{\infty}}(E/\Q)[p]\rightarrow \Sel_{p^{\infty}}(E/\Q_{\infty})^{\Gamma}[p]\rightarrow \left(\ker h'\right)[p]\rightarrow \frac{\Sel_{p^{\infty}}(E/\Q)}{p \Sel_{p^{\infty}}(E/\Q)}.\]
Since it is assumed that $\Sh(E/\Q)[p^{\infty}]$ is finite, it follows that 
\[\Sel_{p^{\infty}}(E/\Q)\simeq \left(\Q_p/\Z_p\right)^{\rank \ E(\Q)}\oplus D,\]where $D$ is a finite group.
It follows that 
\begin{equation}
 \dim_{\F_p} \Sel_{p^{\infty}}(E/\Q)[p]=\dim_{\F_p}\left(\frac{\Sel_{p^{\infty}}(E/\Q)}{p \Sel_{p^{\infty}}(E/\Q)}\right)+\rank \ E(\Q).
\end{equation}
Now,
\begin{align*}
\gp&=\dim_{\F_p}\Sel_{p^{\infty}}(E/\Q_{\infty})^{\Gamma}[p] \\
& \geq \dim_{\FF_p} \Sel_{p^{\infty}}(E/\Q)[p] + \dim_{\FF_p} \left(\ker h'\right)[p] - \dim_{\FF_p} \left(\frac{\Sel_{p^{\infty}}(E/\Q)}{p \Sel_{p^{\infty}}(E/\Q)}\right)\\
& \geq \dim_{\F_p} \left(\ker h'\right)[p]+\rank \ E(\Q),\\
& \geq \dim_{\F_p} \left(\ker h\right)[p],
\end{align*}
where the final inequality is a consequence of \eqref{es2}.
It follows from \cite[Lemma 3.4 and Proposition 3.5]{coates2000galois} that \[\dim_{\F_p} \left(\ker h\right)[p]\geq \mathfrak{c}_E^{(p)},\] and this completes the proof.
\end{proof}

\subsection{} Given a positive integer $n$, we prove an explicit lower bound for the density of elliptic curves over $\Q$ with good ordinary reduction at $p$, for which $\gp\geq n$.
Assume throughout that $\Sh(E/\Q)[p^{\infty}]$ is finite for all elliptic curves $E_{/\Q}$.
\begin{theorem}
\label{th64}
With notation as above,
\[ \begin{split}\mathfrak{d}\left(\mathscr{E}_{g\geq n}\right)>& \frac{1}{\zeta(p)}\times \sum_{\Sigma\in \mathcal{S}_n}\left( \prod_{\ell\in \Sigma} \left( \frac{\ell^8(\ell-1)^2}{(\ell^{10}-1)(\ell^p-1)}\right)\right)\times \left(\frac{p^{8}\#\mathfrak{S}_p}{p^{10}-1}\right)\\
+ & \frac{1}{\zeta(p)}\times\sum_{\Sigma\in \mathcal{S}_{n-1}}\left( \prod_{\ell\in \Sigma} \left( \frac{\ell^8(\ell-1)^2}{(\ell^{10}-1)(\ell^p-1)}\right)\right)\times \left(\frac{p^{8}\#\mathfrak{S}_p'}{p^{10}-1}\right).\end{split}\]
\end{theorem}
\begin{proof}
Recall from the definition that if $E\in \mathscr{E}_{\Sigma}$, then there are at least $\# \Sigma$ many primes (distinct from $2, 3, p$) for which the Tamagawa factors $c_\ell$ are divisible by $p$.

First, for $\Sigma\in \mathcal{S}_n$ and $E\in \mathscr{E}_{\Sigma}$, it follows that $\mathfrak{c}_E^{(p)}\geq n$.
Next, for $\Sigma\in \mathcal{S}_{n-1}$ and $E\in \mathscr{E}_{\Sigma}'$, there are at least $n-1$ of the Tamagawa numbers $c_\ell$ are divisible by $p$ and also, $p|\alpha_p(E)$.
Once again, $\mathfrak{c}_E^{(p)}\geq n$.
Note that the sets $\mathscr{E}_{\Sigma}$ and $\mathscr{E}_{\Sigma}'$ are all mutually disjoint.
Since $E(\Q)[p]=0$ for $100\%$ of elliptic curves, it follows from Lemma \ref{lambda+mu in terms of tamagawa and anomalous} that
\[\mathfrak{d}\left(\mathscr{E}_{g\geq n}\right)\geq \sum_{\Sigma\in \mathcal{S}_n} \mathfrak{d}\left(\mathscr{E}_{\Sigma}\right) + \sum_{\Sigma\in \mathcal{S}_{n-1}} \mathfrak{d}\left(\mathscr{E}_{\Sigma}'\right).\]
The result now follows from Lemma \ref{lemma52}.
\end{proof}
\begin{corollary}
\label{cor to thm64}
Let $n>0$ be an integer and $p$ an odd prime number.
Assume that $\Sh(E/\Q)[p^{\infty}]$ is finite for all elliptic curves $E/\Q$.
The density $\mathfrak{d}(\mathscr{E}_{\mu+\lambda\geq n})$ of the set of elliptic curves $E_{/\Q}$ for which
\begin{enumerate}
 \item $E$ has good ordinary reduction at $p$,
 \item $\mup+\lap\geq n$
\end{enumerate}
is positive.
In particular,
\[
\begin{split}\mathfrak{d}(\mathscr{E}_{\mu+\lambda\geq n})>& \frac{1}{\zeta(p)}\times \sum_{\Sigma\in \mathcal{S}_n}\left( \prod_{\ell\in \Sigma} \left( \frac{\ell^8(\ell-1)^2}{(\ell^{10}-1)(\ell^p-1)}\right)\right)\times \left(\frac{p^{8}\#\mathfrak{S}_p}{p^{10}-1}\right)\\
+ & \frac{1}{\zeta(p)}\times\sum_{\Sigma\in \mathcal{S}_{n-1}}\left( \prod_{\ell\in \Sigma} \left( \frac{\ell^8(\ell-1)^2}{(\ell^{10}-1)(\ell^p-1)}\right)\right)\times \left(\frac{p^{8}\#\mathfrak{S}_p'}{p^{10}-1}\right).\end{split}
\]
\end{corollary}
\begin{proof}
Since $\mup+\lap\geq \gp$, the result follows from Theorem~\ref{th64}.
\end{proof}
\newpage

\section{Tables}\label{tables}
\begin{center}
\begin{table}[h]
\caption{Data for $\mathfrak{S}_p$ for primes $7\leq p<150$.}
\label{tab:1}
\begin{tabular}{ |c|c|c|c| }
\hline
$p$ & $\#\mathfrak{S}_p/p^2$ & $p$ & $\#\mathfrak{S}_p/p^2$ \\ 
\hline
 7 & 0.653061224489796 & 71 & 0.867883356476890 \\
11 & 0.702479338842975 & 73 & 0.932257459185588 \\
13 & 0.781065088757396 & 79 & 0.887357795225124\\
17 & 0.802768166089965 & 83 & 0.898679053563652 \\
19 & 0.789473684210526 & 89 & 0.899886377982578 \\
23 & 0.790170132325142 & 97 & 0.943777234562653 \\
29 & 0.832342449464923 & 101 & 0.911675325948436 \\
31 & 0.842872008324662 & 103 & 0.913375435950608 \\
37 & 0.915997078159240 & 107 & 0.925845051969604 \\
41 & 0.868530636525877 & 109 & 0.945374968437000 \\
43 & 0.874526771227691 & 113 & 0.929751742501370 \\
47 & 0.853779990946129 & 127 & 0.936139872279745 \\
53 & 0.897828408686365 & 131 & 0.897674960666628 \\
59 & 0.866417696064349 & 137 & 0.952847780915339 \\
61 & 0.900295619457135 & 139 & 0.935665855804565 \\
67 & 0.940966807752283 & 149 & 0.933291293184992 \\
 \hline
\end{tabular}
\end{table}
\end{center}

\begin{center}
\begin{table}[h]
\caption{Data for $\mathfrak{S}_p'$ for primes $7\leq p<150$.}
\label{tab:2}
\begin{tabular}{ |c|c|c|c| }
\hline
$p$ & $\#\mathfrak{S}_p'/p^2$ & $p$ & $\#\mathfrak{S}_p'/p^2$ \\ 
\hline
 7 & 0.0816326530612245 & 71 & 0.0208292005554453 \\
11 & 0.0413223140495868 & 73 & 0.0270219553387127 \\
13 & 0.0710059171597633 & 79 & 0.0374939913475405\\
17 & 0.0276816608996540 & 83 & 0.0178545507330527 \\
19 & 0.0581717451523546 & 89 & 0.0222194167403106 \\
23 & 0.0415879017013233 & 97 & 0.0255074928260176\\
29 & 0.0332936979785969 & 101 & 0.00980296049406921 \\
31 & 0.0312174817898023 & 103 & 0.0288434348194929 \\
37 & 0.0306793279766253 & 107 & 0.00925845051969604 \\
41 & 0.0118976799524093 & 109 & 0.0181802878545577 \\
43 & 0.0567874526771228 & 113 & 0.0263137285613595 \\
47 & 0.0208239022181983 & 127 & 0.0169260338520677 \\
53 & 0.0277678889284443 & 131 & 0.0189382903094225 \\
59 & 0.0166618787704683 & 137 & 0.0108689860940913 \\
61 & 0.0349368449341575 & 139 & 0.0142849748977796 \\
67 & 0.0147026063711294 & 149 & 0.0133327327597856 \\
 \hline
\end{tabular}
\end{table}
\end{center}

\section*{Acknowledgements}
DK acknowledges the support of the PIMS Postdoctoral Fellowship. Part of this work was completed when AR was a postdoctoral fellow at the Centre de recherches mathematiques, Montreal. During this time, he was supported by the CRM-Simons postdoctoral fellowship.
AR would like to thank Barry Mazur, Ravi Ramakrishna, Lawrence Washington and Tom Weston for helpful suggestions. The authors thank the referee for the helpful report. 

\bibliographystyle{abbrv}
\bibliography{references}
\end{document}